\documentclass[11pt]{amsart}
 
\usepackage{palatino}
\usepackage{amsfonts,amsmath} 
\usepackage{amsthm}
\usepackage{graphicx,enumerate}
\usepackage[hmarginratio=1:1,hscale=0.75, vscale=0.72]{geometry}
\usepackage[latin1]{inputenc}
\usepackage[all]{xy}

\usepackage{amsmath}
\usepackage{amssymb,mathtools}
\usepackage{stmaryrd}
\usepackage{hyperref}
\usepackage{pdfsync}
\usepackage{color}
\usepackage[T1]{fontenc} 

\usepackage{setspace}
\newtheorem{counter}{Counter}[section]
\newtheorem{lem}[counter]{Lemma}

\newtheorem{thm}[counter]{Theorem}

\newtheorem{prop}[counter]{Proposition}
\newtheorem{cor}[counter]{Corollary}
\newtheorem{conj}[counter]{Conjecture}
\newtheorem{problem}[counter]{Problem}
\numberwithin{equation}{section}

\newcommand{\R}{\mathbb{R}}
 
\renewcommand{\H}{\mathcal{H}} 
\renewcommand{\L}{\mathcal{L}} 
 
\newcommand{\D}{\mathcal{D}}
 
\newcommand{\C}{\mathcal{C}} 
\newcommand{\A}{\mathcal{A}}
\newcommand{\M}{\mathcal{M}}

\newcommand{\lra}{\longrightarrow} 
\newcommand{\Lra}{\Longrightarrow}
\newcommand{\ra}{\rightarrow}

\renewcommand{\ss}{\subset}
 \newcommand{\sse}{\subseteq}
\renewcommand{\~}{\tilde}

\newcommand{\p}{p}

\newcommand{\fal}{\forall}
\newcommand{\8}{\infty} 
\newcommand{\FF}{F_\mu^\p}
\newcommand{\vph}{\varphi}
\newcommand{\vep}{\varepsilon} 

\newcommand{\EE}{E_\mu^{\la,\p}}
\newcommand{\al}{\alpha}
 
\newcommand{\gm}{\gamma}
 \newcommand{\ggm}{\Gamma}
\newcommand{\sm}{\Sigma}
 \newcommand{\om}{\Omega}
\newcommand{\la}{\lambda}

\renewcommand{\d}{\!\operatorname{d}\!} 

\newcommand{\disp}{\displaystyle}

\newcommand{\supp}{\operatorname{supp}}
\newcommand{\diam}{\operatorname{diam}}
\newcommand{\argmin}{\operatorname{argmin}}
\newcommand{\Area}{\operatorname{Area}}
\newcommand{\conv}{\operatorname{Conv}}
 
 \renewcommand{\D}{\diam\supp(\mu)}
\newcommand{\te}{\textrm}
\newcommand{\be}{\begin{equation}}

\newcommand{\ee}{\end{equation}}
{   \end{list} }

\definecolor{mygreen}{rgb}{0.1,0.75,0.2}

\title{Average-distance problem for parameterized curves}
\author{\sc{Xin Yang Lu and Dejan Slep\v{c}ev}}
\address{
Department of Mathematical Sciences, Carnegie Mellon University, Pittsburgh, PA, 15213, USA. \\
tel. +412 268-2545, 
emails: xinyang@andrew.cmu.edu, slepcev@math.cmu.edu}

\date{\today}

\begin{document}
  
\begin{abstract}
We consider approximating a measure by a parameterized curve subject to length penalization.
That is for a given finite positive compactly supported measure $\mu$,
for $p \geq 1$ and $\lambda>0$ we consider the functional
\[ E(\gamma) = \int_{\mathbb{R}^d} d(x, \Gamma_\gamma)^p d\mu(x) + \lambda \,\textrm{Length}(\gamma) \]
where $\gamma:I \to \mathbb{R}^d$, $I$ is an interval in $\mathbb{R}$, $\Gamma_\gamma = \gamma(I)$, and $d(x, \Gamma_\gamma)$ is the distance of $x$ to $\Gamma_\gamma$.

The problem is closely related to the average-distance problem, where the admissible class are the connected  sets of finite Hausdorff measure $\mathcal H^1$, and to (regularized) principal curves studied in statistics. We
obtain regularity of minimizers in the form of  estimates on the total curvature of the minimizers.
We prove that for measures $\mu$  supported in two dimensions   the minimizing curve is injective if $p \geq 2$ or if $\mu$ has bounded density. This establishes that the minimization over parameterized curves is equivalent to minimizing over embedded curves and thus confirms that the problem has a geometric interpretation. 
\end{abstract}


\maketitle


\textbf{Classification. }
49Q20, 
 49K10, 
 49Q10, 
  35B65 

\section{Introduction}
Approximating measures by one dimensional objects arrises in several fields.
In the setting of optimization problems connected to network planning (such as for urban transportation network) and irrigation it was introduced by {Buttazzo, Oudet} and {Stepanov} \cite{BOS}, 
and has been extensively studied \cite{ BPS09,  BAdd1, BAdd2, BAdd3,  BS1, BS2, Lem, LuSl, SanTil}.

In this setting the problem is known as the \emph{average-distance problem}.
Given a set $\sm \subset \R^d$ let $d(x, \sm) = \inf_{y\in\sm} |x-y|$. Let $\M$ be the set of positive, finite compactly supported measures in $\R^d$ for $d \geq 2$, with $\mu(\R^d)>0$.
\begin{problem}\label{main pen}
Given measure $\mu \in \M$, and parameters $\p\geq 1$, $\lambda>0$,
we consider the average-distance problem in the penalized (as opposed to constrained) form: Minimize
\begin{equation*} 
G_\mu^{\la,\p}(\sm):= \int_{\R^d} d(x, \sm)^\p \,\d\mu(x)+\lambda \H^1(\sm),
\end{equation*}
with the unknown $\sm$ varying in the family
$$\A:=\{\sm\sse\R^d: \sm\ \text{compact, path-wise connected and } \H^1(\sm)<\8 \}.$$
\end{problem}

Another application in which a measure is to be approximated by a one-dimensional object
arises in machine learning and statistics where one wishes to obtain the curve that best represents  the data given by a (probability) measure $\mu$. The problem in this setting was introduced by
Hastie \cite{Add1} and  Hastie and Stuetzle \cite{Add2}, and its solution is known as the (regularized) principal curve. A variant of the  problem can be formulated as follows:  let
\[ \C:=\{\gm:[0,a]\lra \R^d: a\geq 0, \gm {\text{ is Lipschitz with }} 
|\gm'| \leq 1 \;\;  \text{a.e.}\}. \]
For given $\gm\in \C$, we define its length, $L(\gm)$, as its total variation $\|\gm\|_{TV([0,a])}$.
Furthermore given $\gm\in \C$ we define $\ggm_\gm = \gm([0,a])$.
The problem can be stated as follows:
 \begin{problem}\label{main1}
 Given a  measure $\mu \in \M$, parameters $\lambda>0$, $\p\geq 1$ find 
$\gm\in \C$ minimizing 
$$E_{\mu}^{\lambda,\p}(\gm):= \int_{\R^d} d(x, \Gamma_\gm)^p d \mu(x) +\la L(\gm).$$
\end{problem}
We remark that in machine learning the problem has been considered most often with $p=2$, with a variety of regularizations, as well as with length constraint (instead of length penalization) \cite{KKLZ, Smola, Tib}.
Regularizing with a length term is the lowest order (in other words the weakest) of regularizations considered. 
We note that the first term of energy measures the approximation error while the second therm penalizes the complexity of the approximating object (curve).

The existence of minimizers of Problem 1.2  is straightforward to establish in the class of parameterized curves. However it is not clear if for a general measure $\mu$ the minimizing curve is injective, in other words it may 
have self-intersections and not be an embedded curve. Here we show that in two dimensions if $p \geq 2$ then the minimizer in fact is an injective curve. We also show that if $\mu$ has bounded density with respect to Lebesgue measure then the minimizer is an injective curve for all $ 1 \leq p < \infty$. More precisely the main result of our work is:
\begin{thm}\label{main}
Consider dimension $d=2$.
 Let $\mu \in \M$ and let  $\lambda >0$ and $p \geq 1$. If $p <2$ assume that $\mu$ is absolutely continuous with respect to Lebesgue and that its density, $\rho$, is bounded.
 Let $\gamma:[0,L] \to \R^2$ be an arc-length-parameterized minimizer of $E_\mu^{\la,\p}$. Then $\gm$ is injective and  in particular $\Gamma_\gm$ is a curve embedded in $\R^2$.
\end{thm}
The theorem implies that the problem can also be posed as a minimization problem among embedded curves. 
We note that, as we discuss at the beginning of Section 4, the conclusion of the theorem holds for all $1 \leq p <\infty$ if $\mu$ is a discrete measure. 

We hypothesize that the range of $p$ in the theorem is sharp:
\begin{conj}
For  $1 \leq p < 2$  there exist $\la>0$ and measure $\mu$ for which the global minimizer is not injective.
\end{conj}

Further relevant question is the regularity of minimizing curves. We note that in \cite{Sl13} it was shown that 
minimizers of the Problem 1.1, even for measures with smooth densities, can be embedded curves which have corners. Since these are also minimizers of Problem 1.2, we conclude that minimizers of Problem 1.2 are not $C^1$ curves in general. 

Thus we consider regularity of minimizers in the sense of obtaining estimates on
the total variation of $\gm'$, where $\gm$ is an arc-length-parameterized minimizer. 
This allows us to consider the curvature as a measure and provides bounds on the total curvature of a
segment of the minimizing curve in terms of the mass projecting on the segment.
To do so we use techniques developed in \cite{LuSl}.

\medskip

\newpage

This paper is structured as follows:
\begin{itemize}
\item in Section \ref{prel} we present preliminary notions and results, and
prove existence of minimizers of Problem \ref{main1}. We furthermore show that the minimizers are contained in the convex hull of the support of the measure $\mu$.

\item in Section \ref{inj} we prove the injectivity of minimizers (Theorem \ref{main})
in the two dimensional case.

\item in Section \ref{app} we extend the regularity estimates of \cite{LuSl} to $p>1$ and prove them in the setting of parameterized curves.  We furthermore provide the version of estimates in $\R^2$ which roughly speaking bounds how much a minimizer can turn to the left by the mass to the right of the curve. This is a key result needed to prove injectivity. 
\end{itemize}

\section{Preliminaries}\label{prel}

In this section we provide some preliminary results including the 
proof of existence of minimizers
of Problem \ref{main1} (Lemma \ref{exist}).

We define the distance between curves in $\C$ as follows: Let $\gm_1, \gm_2 \in \C$ with domains
$[0,a_1]$, $[0, a_2]$ respectively. We can assume that $a_1 \leq a_2$.
Let $\tilde \gm_1 : [0, a_2] \to \R^d$ be  the extension of $\gm_1$ to $[0,a_2]$ as follows
\begin{equation} \label{curv_ext}
\tilde \gm_1(t) = \begin{cases}
\gm_1(t) \quad & \te{if } t \in [0, a_1]  \\
\gm_1(a_1) & \te{if } t \in (a_1, a_2].
\end{cases} 
\end{equation}
Let
\[ d_\C(\gm_1, \gm_2) = \max_{t \in [0,a_2]} |\tilde \gm_1(t)-\gm_2(t)|. \]

The first issue is the existence of minimizers. A preliminary lemma is required.
Given a measure $\mu \in \M$, 
and $\p\geq 1$, let
$$F_\mu^\p:\A\lra [0,\8),\qquad F_\mu^\p(\sm):= \int_{\R^d} d(x, \sm)^\p \, \d\mu(x).$$
\begin{lem}\label{comp}
Given a measure $\mu \in \M$, parameters $\la>0$, $\p\geq 1$,
 then for any minimizing 
sequence $\{\gm_n\}$ of Problem \ref{main1} it holds:
\begin{itemize}
\item length estimate: 
\begin{equation*}
\limsup_{n \to \infty} L(\gm_n) \leq  \frac{1}{\la} \big(\!\diam\supp(\mu)\big)^p,
\end{equation*}
\item confinement condition: there exists
a compact set $K\sse\R^d$ such that $\ggm_{\gm_n}\sse K$ for all $n$.
\end{itemize}
\end{lem}

\begin{proof}
Boundedness of the length is obtained by using a singleton as a competitor. 
Fix an arbitrary point $z\in \supp(\mu)$, and
let $\gamma:[0,0] \lra \{z\}$. Then
\begin{equation}\label{inf E}
\inf_\C\EE\leq \EE(\gamma)\leq \int_{\R^d} |x-z|^\p\d\mu(x)\leq \big(\!\diam\supp(\mu)\big)^\p.
\end{equation}
Since $\{\gm_n\}$ is a minimizing sequence, \eqref{inf E}
gives
\begin{equation}\label{max length}
(\fal \vep) (\exists n_0) 
(\fal n\geq n_0) \qquad 
\la L(\gm_n)\leq\EE(\gm_n)\leq \big(\!\diam\supp(\mu)\big)^\p+\vep.
\end{equation}
To prove the confinement condition, note that
for any $r\geq 0$, $\gamma\in \C$ it holds
$$\ggm_\gamma\cap \big(\!\supp(\mu)\big)_r=\emptyset \;
\Lra  \; \EE(\gamma)\geq \FF(\ggm_\gamma)\geq r^\p,$$
where $\big(\!\supp(\mu)\big)_r=\{x\in\R^d:\inf_{y\in \supp(\mu)}|x-y|\leq r \}$.
Thus \eqref{inf E} gives
\begin{equation*}
(\fal \vep) (\exists n_0) : \quad
(\fal n\geq n_0) \quad 
 \FF(\ggm_\gamma)\leq\EE(\gm_n)\leq \big(\!\diam\supp(\mu)\big)^\p+\vep,
\end{equation*}
and combining with length estimate \eqref{max length} and taking $\vep=1$ gives
\begin{equation*}
(\exists n_0) 
(\fal n\geq n_0) \quad 
\ggm_{\gm_n} \sse \big(\!\supp(\mu)\big)_{(\D+(\D)^\p/\la+1+1/\la)}
\end{equation*}
concluding the proof.
\end{proof}

Given a measure $\mu \in \M$ and curve $\gamma$,
let $\pi$ be a probability measure supported on $\R^d \times \Gamma_\gamma$ such that
the first marginal of $\pi$ is $\mu$ and that for $\pi$-a.e. $(x,y)$, $|x-y| = \min_{z \in \Gamma_\gm} |x-z|$.
The existence of such a measure is proved in Lemma 2.1 of \cite{LuSl}. 
Let $\sigma$ be the second marginal of $\pi$. Then $\sigma$ is supported on $\Gamma_\gm$ and $\pi$ is an optimal transportation plan between $\mu$ and  $\sigma$ for the cost $c(x,y) = |x-y|^q$, for any $q \geq 1$.
In other words $\sigma$ is a projection of $\mu$ onto $\Gamma_\gamma$. 

We remark that in \cite{ManMen} it has been proven that for any $\sm \in \A$,
the ridge 
$$\mathfrak{R}_\sm:=\{x: \text{ there exist } p,q\in \sm, \ 
p\neq q,\ |x-p|=|x-q|=d(x,\sm)\}$$x
is $\H^{1}$-rectifiable. Thus for any $\sm\in \A$
the (point-valued) ``projection'' map
\begin{equation}\label{proj defn}
\Pi_\sm:\R^d\lra \sm,\quad \Pi_\sm(x) := \text{the point of } \sm 
\text{ such that } |x-\pi(x)|=d(x,\sm)
\end{equation}
is well defined $\L^2$-a.e. Consequently if $\mu$ is absolutely continuous with respect to Lebesgue measure the measures $\pi$ and $\sigma$ above are uniqeuly defined and furthermore $\sigma = \Pi_{\Gamma_\gm\, \sharp} \mu$.


\begin{lem}\label{exist}
Consider  a positive measure $\mu$ and parameters $\la>0$, $\p\geq 1$.
Problem \ref{main1} has a minimizer $\gm \in \C$. Furthermore for any minimizer
$\Gamma_\gm$ is contained in $\conv(\mu)$, the  convex hull of the support of $\mu$.
\end{lem}
We note that since the energy is invariant under reparameterizing the curve it follows that the problem has a minimizer $\gamma \in \C$ which is arc-length parameterized. 
\begin{proof}
Consider a minimizing sequence $\{\gm_n\}$ in $\C$.
Since a reparameterization does not change the value of the functional we can 
assume that $\gm_n$ are arc-length parameterized for all $n$.
Lemma \ref{comp} proves that $\{\gm_n\}$ are uniformly bounded and have uniformly bounded lengths. 
Let $L$ be the supremum of the lengths and let $\tilde \gm_n$ be the extensions of the curves as in
\eqref{curv_ext} to interval $[0,L]$. The curves $\{\tilde \gm_n\}$ 
satisfy the conditions of Arzel\`a-Ascoli  Theorem.
Thus, along a subsequence (which we assume to be the whole sequence) they converge uniformly (and thus in $\C$) to a curve $\gm:[0,L] \to \R^d$. Since all of the curves are 1-Lipschitz, so is $\gm$ and thus it belongs to $\C$. 

 Since $\vph\mapsto\FF(\ggm_\vph)$ is continuous  and
$\vph\mapsto L(\vph)$ is lower-semicontinuous with respect to the convergence in $\C$,
it follows $\liminf_{n \to \infty} \EE(\tilde \gm_n) \geq \EE(\gm)$.
Since $\{\tilde \gm_n\}_n$ is also a minimizing sequence, $\gm$ is a minimizer of $\EE$.

Now we prove that any minimizer is contained in the convex hull of $\supp(\mu)$. The argument relies on fact that the projection to a convex set decreases length, which we state in Lemma \ref{conv_proj} below.
Let $\gm \in \C$ be a minimizer of $\EE$. Assume  it is not contained in the convex hull, $K = \conv(\mu)$, of the support of $\mu$. Then there exists $T \in [0,a]$ such that $\gm(T) \not\in K$. 
Let $[t_1, t_2]$ be the maximal interval such that $\gm((t_1, t_2)) \cap K = \emptyset$.
We claim that $\sigma((t_1,t_2)) = 0$. Otherwise consider $\tilde \gm$ be the projection of $\gm$ onto $K$. The distances between $\tilde \gm(t)$ and points in $K$ are strictly less than the distances between $\gm(t)$ and the points in $K$ and thus $F_\mu^\p(\ggm_{\tilde{\gm}})< F_\mu^\p(\ggm_\gm)$. 
By Lemma \ref{conv_proj} the length of $\tilde \gm$ is less than or equal to the length of $\gm$. Consequently $\EE(\tilde \gm) < \EE(\gm)$, which contradicts the assumption that $\gm$ is a minimizer. 
Thus  $\sigma((t_1,t_2)) = 0$.

If $\gm(t_1)$ and $\gm(t_2)$ belong to $K$ then consider $\gm_2$ obtained by replacing the segment
$\gm|_{[t_1, t_2]}$ of $\gm$ by a straight line segment. Note that the length of $\gm_2$ is less than the length of $\gm$ (since otherwise $\gm|_{(t_1, t_2)}$ would have to be a line segment which contradicts the fact that it is outside of $K$). Also note that $F_\mu^\p(\ggm_{\gm_2}) = F_\mu^\p(\ggm_\gm)$ and thus 
$\EE(\gm_2) < \EE(\gm)$ , which contradicts the assumption that $\gm$ is a minimizer. 

If $\gm(t_1) \not \in K $ then $t_1 = 0$. Noting that $\gm_2 = \gm|_{[t_2, a]}$ has lower energy than $\gm$ contradicts the minimality of $\gm$. The case  $\gm(t_2) \not \in K $ is analogous. 
\end{proof}

\begin{lem} \label{conv_proj}
Given a convex set $\om$, let $\Pi_\om : \R^d \to \overline \om$ be the projection onto $\overline \om$ defined by $\Pi_\om(x) = \argmin_{z \in \overline \om} |x-z|$. Let  $\gm: [a,b] \lra \R^d$ be a rectifiable curve. Then
\[ \H^1 (\Pi_\om (\gm(I))) \leq  \H^1 (\gm(I) ) . \]
\end{lem}
\begin{proof}
It is well known that projection to a convex set is a 1-Lipschitz mapping, see Proposition 5.3 in the book by Brezis \cite{Brezis}. That is  for all $x,y \in \R^d$
\[ |\Pi_\om(x) - \Pi_\om(y)| \leq |x-y|. \]
Furthermore equality holds only if $x - \Pi_\om(x) = y - \Pi_\om(y)$. 
The claim of the lemma readily follows. 
\end{proof}
%

\section{Injectivity}\label{inj}

The main aim of this section is to prove injectivity for 
minimizers of Problem
\ref{main1} in two dimensions.
We say that $P \in \Gamma_\gm$ is a \emph{double point} if $\gm^{-1}(P)$ has at least two elements. 
Our goal is to show that there are no double points. 
%
 Note that if $\ggm_\gm$ is a simple curve, then it admits
 an injective parameterization, which is shorter than any noninjective parameterization. Thus in the following we will
 consider only minimizers containing points with order at least 3, that is points $P$ such than for $r>0$ small enough $(\Gamma_\gm \cap B(P,r)) \backslash \{P\}$ has at least three connected components. 
 
\begin{lem}\label{non zero angle}
 Let $\mu\in \M$ and let $\lambda>0$ and $p \geq 1$. Let $\gm:[0,L] \to \R^d$ be an arc-length-parameterized minimizer of $\EE$. 
Assume there exist  times $0<t<s<L$ such that $\gm(t)=\gm(s)$.
Then $\gm$ is differentiable at $t$ and at $s$.

Furthermore $\gm'(t)= \gm'(s)$ or $\gm'(t)= -\gm'(s)$.
\end{lem}
\begin{proof}
Assume the claim does not hold. Without a loss of generality we can assume that $\gamma$ is not differentiable at $s$.
Then there exist sequences 
$\{s_n^-\}\searrow s$, $\{s_n^+\}\nearrow s$, such that 
$\angle \gm(s_n^-)\gm(s)\gm(s_n^+)\ra\al  < \pi$ as $n \to \infty$. 
Note that by Lemma \ref{comp} 
for all $z \in \supp(\mu)$ and all $y \in \Gamma_\gm$, $|z-y| < \diam(\supp(\mu))$.
\begin{figure}[h!]
\includegraphics[scale=1.1]{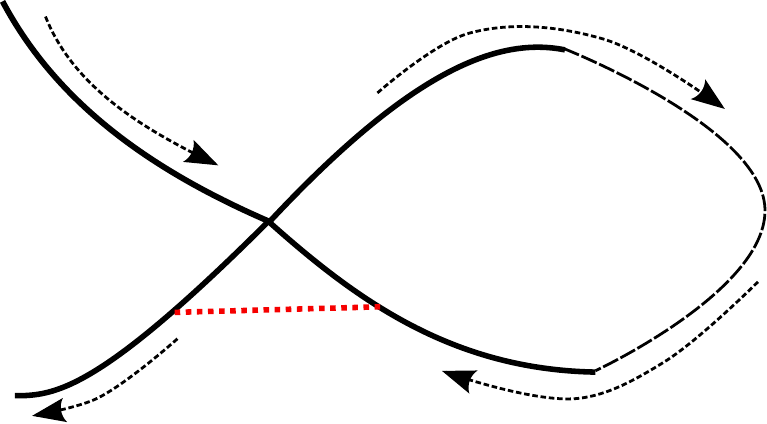}
\put(-127,43){$\gm(s_n^-)$}
\put(-209,43){$\gm(s_n^+)$}
\put(-151,63){$\gm(s)=\gm(t)$}
\caption{This is a schematic representation of the variation. 
The black lines belong to the (graph of) $\gm$, while the red
dotted line belongs to the (graph of) competitor $\~\gm_n$. Time increases
along the direction of the arrows.}
\end{figure}

Consider the 
competitors $\~\gm_n$ constructed in the following way:
Let
\begin{alignat*}{2}
\xi_n^* & :[0,1] \lra \R^d, & \qquad \xi_n^*(u)& :=(1-u)\gm(s_n^-)+ u\gm(s_n^+), \\
\xi_n &:\big[0,|\gm(s_n^-)-\gm(s_n^+)|\big]\lra \R^d, & \qquad
 \xi_n(u)&:= \xi_n^*\big(u/|\gm(s_n^-)-\gm(s_n^+)|\big).
\end{alignat*} 
 Let $\~\gm_n:\big[0,L(\gm)-(s_n^+-s_n^-)+|\gm(s_n^-)-\gm(s_n^+)| \big]$, where
$$ 
\qquad
\~\gm_n(u):=
\left\{
\begin{array}{cl}
\gm(u) & \text{if } u\leq s_n^-,\\
\xi_n(u-s_n^-) & \text{if } s_n^-\leq u\leq s_n^-+|\gm(s_n^-)-\gm(s_n^+)|,\\
\gm\big(u-s_n^--|\gm(s_n^-)-\gm(s_n^+)|+s_n^+\big)  & \text{if }  u
\geq s_n^-+|\gm(s_n^-)-\gm(s_n^+)|.
\end{array}
\right. $$

Since by hypothesis $\{\angle \gm(s_n^-)\gm(s)\gm(s_n^+)\}\ra \al\neq 0$, it follows 
(for any sufficiently large $n$)
$$|s_n^+-s_n^-|- |\gm(s_n^-)-\gm(s_n^+)|  \geq  c  |s_n^+-s_n^-|,$$ 
for some constant $c>0$ independent of $n$.
Hence
\begin{equation}\label{ineq length}
L(\gm)\geq L(\~\gm_n)+ c  |s_n^+-s_n^-|.
\end{equation}
By taking $n$ large we can assume that $ |s_n^+-s_n^-| < 1$.

We claim that 
\begin{equation}\label{ineq energy}
F_\mu^\p(\~\gm_n)-F_\mu^\p(\gm)\leq 
\mu\left(\left\{z:\argmin_{y\in \ggm_\gm}|z-y| \sse \gm\big((s_n^-,s_n^+)\backslash \{s\}\big)
\right\} \right) \p D^{\p-1} |s_n^+-s_n^-|,
\end{equation}
where $D:=1+ \diam\supp(\mu)$.
Note that if a point $z$ satisfies
$$d(z,\ggm_\gm) < d(z,\ggm_{\~\gm_n}) $$
then $\argmin_{y\in \ggm_\gm}|z-y| \sse \gm\big((s_n^-,s_n^+)\backslash \{s\}\big)$. 

The constant $\p D^{\p-1}$ is due to the fact that any such point
$z\in \supp(\mu)$ satisfies, due to Lemma \ref{exist} and construction of $\tilde \gm_n$,
$$\max\{d(z,\ggm_\gm),d(z,\ggm_{\~\gm_n}) \}\leq D.$$
By construction there exists a point $z_n'\in \ggm_{\~\gm_n}\backslash
\ggm_\gm$ satisfying $|z-z_n'|=d(z,\ggm_{\~\gm_n})$. Denoting by
$z'\in\gm\big((s_n^-,s_n^+)\backslash \{s\}\big)$ a point satisfying
$|z-z'|=d(z,\ggm_\gm)$, we conclude
$$\big||z-z'|^p-|z-z_n'|^p\big|\leq \p D^{\p-1} |s_n^--s_n^+|. $$
Since for sufficiently large $n$ it holds
$$\mu\left(\left\{z:\argmin_{y\in \ggm_\gm}|z-y| \sse \gm\big((s_n^-,s_n^+)\backslash \{s\}\big)
\right\} \right)= o(|s_n^+-s_n^-|),$$
 combining with \eqref{ineq length} gives that the minimality of
$\gm$ is contradicted by $\~\gm_n$ for sufficiently large $n$.
\medskip

To show the second claim assume that  $\gm'(t) \neq  \gm'(s)$ and $\gm'(t) \neq -\gm'(s)$.
Consider the following "reparameterization  " of the curve $\gm$. Let $\tilde \gamma : [0,L] \to \R^d$ be defined by
\[ \tilde \gm(r) = \begin{cases}
\gm(r) \quad & \te{for } r \in [0,t] \\
\gm(s-(r-t)) & \te{for } r \in (t,s] \\
\gm(r) & \te{for } r \in (s,L]. 
\end{cases}\]
Then $\tilde \gm$ is also a minimizer of $\EE$. However $\tilde \gm$ is not differentiable at $t$ (and at $s$), which contradicts the first part of the lemma.
\end{proof}

\medskip

\begin{proof}[Proof of Theorem \ref{main}]
Let $\gm: [0,L] \to \R^2$ be an arc-length-parameterized minimizer of $\EE$, and let $\Gamma_\gm = \gm([0,L])$. Recall that $P \in \Gamma_\gm$ is a double point if $\gm^{-1}(P)$ has at least two elements. 
Our goal is to show that $\Gamma_\gm$ has no double points. 

We claim that there exists $\delta_1 \in (0,1)$ such that $\gm$ is injective on  $[0, \delta_1)$.
The argument by contradiction is straightforward, by considering $\gm$ restricted to $[ \delta_1, L]$ for $\delta_1>0$ small to be a competitor.
Likewise for $\delta_1$ small $\gm((L - \delta_1, L])$ has no double points.

Assume that there are double points on $\gamma([\delta_1, L-\delta_1])$.
Let $t_2 = \sup\{t < L - \delta_1 \::\: \gamma(t)$ is a double point$\}$.
Note that $\gm$ is injective on $(t_2, L]$. 
We claim that $\gm(t_2)$ is a double point. 
Assume it is not. Then there exist increasing sequences $s_k < r_k < t_2$ converging to $t_2$ such that $\gm(s_k) = \gm(r_k)$. By considering their subsequences we can assume that $r_k < s_{k+1}$ for all $k$. Then the intervals $[s_k, r_k]$ are all mutually disjoint. Since $\gm(s_k) = \gm(r_k)$ we conclude that $\|\gm \|_{TV(s_k, r_k)} \geq \pi$ which implies that $\| \gm \|_{TV([0,L])}$ is infinite. This contradicts the regularity estimate of Proposition \ref{reg_p}. Thus $\gm(t_2)$ is a double point.
Hence there exists $ t_1\in (\delta_1, t_2)$ such that $\gm(t_1) = \gm(t_2)$.
By Lemma \ref{non zero angle}, 
there are two possibilities: either $\gm'(t_1) = \gm'(t_2)$ or $\gm'(t_1) = -\gm'(t_2)$. Since the arguments are analogous we assume  $\gm'(t_1) = \gm'(t_2)$.
By regularity of $\gm$ established in \eqref{regest2}, there exists $\delta_2 \in (0, \delta_1)$ such that 
$\| \gm' \|_{TV(t_1, t_1+ \delta_2)} < \frac{1}{8}$ and $\| \gm' \|_{TV(t_2, t_2+ \delta_2)} < \frac{1}{8}$. Therefore $\gm$ restricted to $[t_1, t_1+\delta_2]$ is injective. Since $\gm((t_1, t_1+ \delta_2))$ has no double points $\gm((t_1, t_1+ \delta_2)) \cap \gm((t_2, t_2+ \delta_2)) = \emptyset$.

We can assume without a loss of generality that $\gm(t_1) = 0$ and  $\gm'(t_1) = e_1$. 
The bound on total variation of $\gm'$ above implies that 
$\gm' \cdot e_1 > \frac{7}{8}$ on the intervals considered.
Therefore we can reparameterize the curve using the first coordinate as the parameter. That is
there exists Lipschitz functions $x, \alpha, \beta : [0, \frac{7}{8}\delta_2) \to \R$ such that $\frac{7}{8}  < x'(s) \leq 1$ a.e. and for all $s \in [0, \delta_2]$
\[ \gm_1(t_1+s) = (x(s), \alpha(x(s))) \quad \te{ and } \quad  \gm_1(t_2+s) = (x(s), \beta(x(s))). \]
Let $\delta = \delta_2/3$.
Without a loss of generality we can assume that $\alpha > \beta$ on $(0, \delta)$.

Our goal is to arrive at contradiction by showing that  $\alpha = \beta$ on some interval $[0, \tilde \delta)$. 
 The reason is that $\alpha$ cannot separate from $\beta$ is that for $\alpha$ to turn upward, by Lemma \ref{reg_U}, there must be mass beneath $\alpha$ talking to that part of the curve. But the mass beneath $\alpha$ which talks to $\alpha$ must lie above $\beta$. However the region between $\alpha$ and $\beta$ cannot contain enough mass to allow for the needed turn. Below we make this argument precise. 
For a.e. $x \in [0, \delta)$, $\alpha$ and $\beta$ are differentiable at $x$ and we define $\ell_\alpha^+(x) = \{ (x, \alpha(x)) + r (\alpha'(x),1) \::\: r \geq 0 \}$ to be the halfline perpendicular to $\alpha$ at $x$ extending above $\alpha$ and 
$\ell_\alpha^-(x) = \{ (x, \alpha(x)) + r (\alpha'(x),1) \::\: r \leq 0 \}$ the halfline below, as illustrated on Figure \ref{fig:ab}. The halflines
$\ell_\beta^+(x)$ and $\ell_\beta^-(x)$ are defined analogously. 
\begin{figure}[h!]
\includegraphics[scale=1.1]{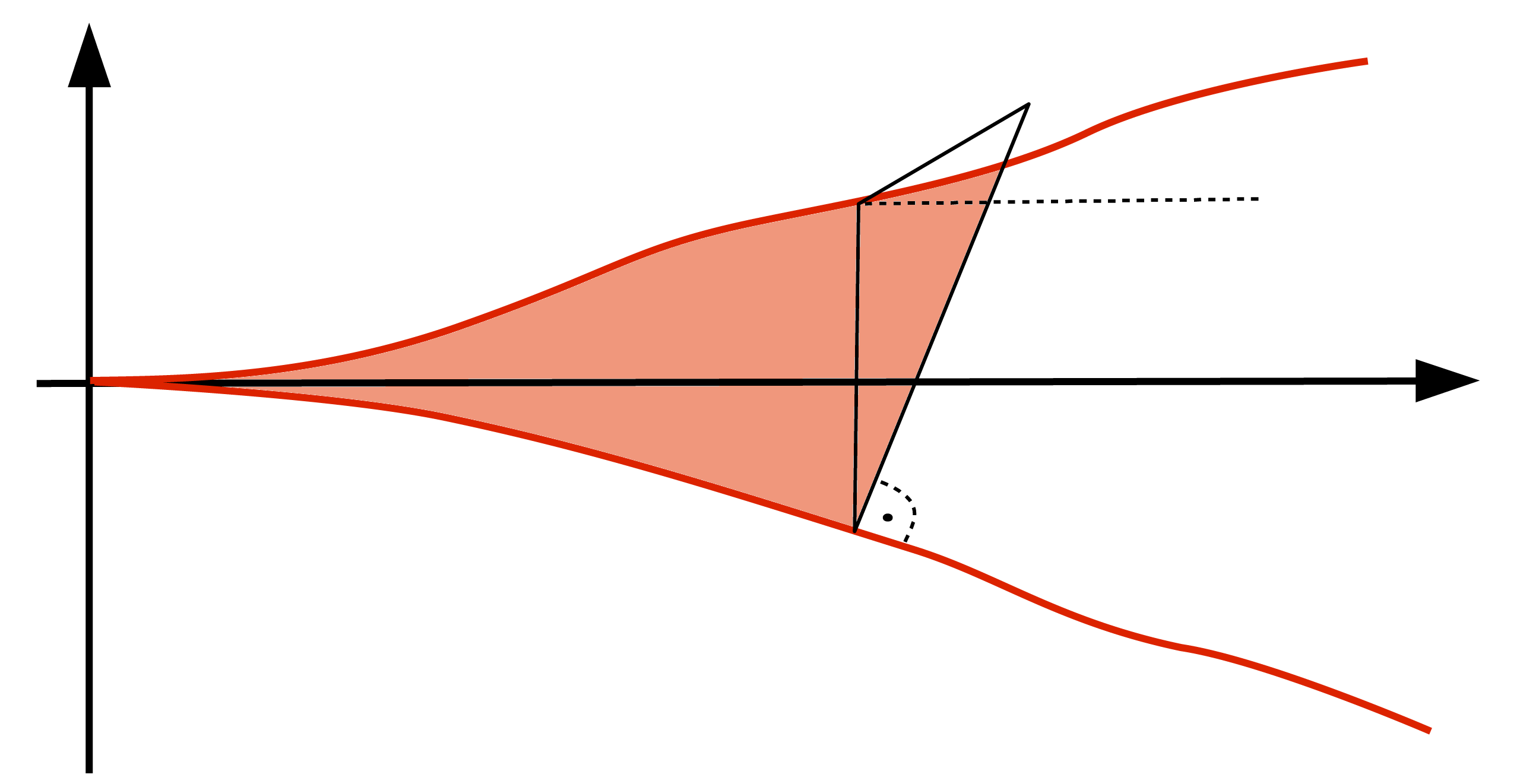}
\put(-156,93){$x$}
\put(-206,94){\Large $A_\beta(x)$}
\put(-125,103){$\ell_\beta^+(x)$}
\put(-154,45){$P$}
\put(-154,134){$Q$}
\put(-109,153){$Z$}
\put(-50,149){\Large $\alpha$}
\put(-50,26){\Large $\beta$}
\put(-148,72){\small $\theta$}
\put(-125,136){\small $\xi$}
\put(-137,140){\small $\bar \ell$}
\caption{The geometry of the configuration near the last double point.}
\label{fig:ab}
\end{figure}

 Let $S_\alpha(x)= \{(z, \alpha(z)) \: : \: z \in [0,x]\}$ and $S_\beta(x)= \{(z, \beta(z)) \: : \: z \in [0,x]\}$.
Let $ U_\alpha(x)$ be the connected component of $\R^2 \backslash (S_\alpha(x) \cup \ell_\alpha^-(x) \cup S_\beta(\delta))$ containing the point $(x/2, (\alpha(x/2)+\beta(x/2))/2)$.
Analogously we define $A_\beta(x)$  be the connected component of $\R^2 \backslash (S_\alpha(\delta) \cup \ell_\beta^+(x) \cup S_\beta(x))$ containing the point $(x/2, (\alpha(x/2)+\beta(x/2))/2)$.
Note that all of the mass below the curve $\alpha$ and talking to $S_\alpha(x)$ is a subset of $U_\alpha(x)$. Likewise all of the mass above the curve $\beta$ and talking to $S_\beta(x)$ is a subset of $A_\beta(x)$.

We introduce:
 \begin{equation} \label{fg}
 f(x) = \sup_{0 \leq z \leq x}  \frac{\alpha'(z)}{|(1, \alpha'(z))|} \quad \te{ and } \quad
 g(x) = \inf_{0 \leq z \leq x}  \frac{\beta'(z)}{|(1, \beta'(z))|}.  
 \end{equation}
 Note that $f(x) - g(x) > 0$ on $(0, \delta)$ and that, using the assumption on total variation of $\gm'$, it follows that for $x \in [0, \delta)$, $|f(x)|$ and $|g(x)|$ are less than $\frac{1}{8}$. 

 Let $\bar \ell$ be the line passing through $Q=(x, \alpha(x))$ with slope is $\frac18$. It stays above the graph of $\alpha$ on $(x, \delta)$. Let $Z$ be the intersection point of $\bar \ell$ and $\ell_\beta^+(x)$. We note that $\theta < \arctan(\frac{1}{8})$ and thus $[QZ]$ is the shortest side of triangle $\Delta PZQ$. Therefore
\[ |PZ| < 2 |QP| = 2 (\alpha(x) + \beta(x)) \leq 2 x (f(x) - g(x)). \]
Since $f(x) - g(x)$ is an increasing function we conclude that 
\[ \sup_{z \in A_\beta(x)} d(z,\Gamma_\gm) \leq    2 x (f(x) - g(x)). \]
Likewise 
\[ \sup_{z \in U_\alpha(x)} d(z,\Gamma_\gm) \leq    2 x (f(x) - g(x)). \]
Lemma \ref{reg_U} implies that 
\begin{equation} \label{test1}
f(x) \leq  \frac{p}{\lambda}  (2x(f(x) - g(x))^{p-1}   \mu(U_\alpha(x)) \quad \te{ and } \quad
g(x)  \geq  \frac{p}{\lambda} (2x(f(x) - g(x))^{p-1}  \mu(A_\beta(x)).
\end{equation}
We first focus on $p \geq 2$. From the above inequalities it  follows that for some constant $c>0$
\[ ( f(x) - g(x))^{2-p} \leq c x^{p-1} . \]
Since as $x \to 0^+$ the left-hand side remains bounded from below while the right-handside converges to zero we obtain a contradiction, as desired.

We now consider the more delicate case: $1 \leq p <2$. Recall that we now assume that $\mu$ has bounded density $\rho$.
To obtain a bound on $\mu(A_\beta(x))$ we estimate the area of $A_\beta(x)$.
The area of $A_\beta(x)$ is bounded from above by the sum of the areas of the region between the curves to the left of line segment $[PQ]$ and the area of triangle $PZQ$ on Figure \ref{fig:ab}.

We note that $\theta < \arctan(\frac{1}{8})$ and the angle $\angle ZQP$ is $\frac{\pi}{2} +  \arctan(\frac{1}{8})$. Therefore $\xi > \frac{\pi}{6}$. 
Using the law of sines and $\alpha(x) - \beta(x) \leq x(f(x) - g(x))$ we obtain
\[ \frac{1}{2x(f(x) -g(x))} < \frac{\sin \xi}{\alpha(x) - \beta(x)} = \frac{\sin \theta}{|QZ|} < \frac{1}{8 |QZ|}. \]
Therefore
\[ \Area(\Delta PZQ) \leq \frac12 |QP| \cdot |QZ| \leq \frac{1}{8}  x^2 (f(x)-g(x))^2 . \]
Consequently
 \[ \Area(A_\beta(x)) \leq x^2 (f(x)- g(x)) +   x^2 (f(x)-g(x))^2  \leq 2 x^2 (f(x)- g(x)). \]
 Same upper bound holds for $ \Area(U_\alpha(x))$. 
 Therefore 
 \[ \max\{ \mu(U_\alpha(x)) ,\mu(A_\beta(x)) \} \leq 2 \| \rho \|_{L^\infty} x^2 (f(x)-g(x)). \]
 Combining with estimate \eqref{test1}  and using that $2x(f(x) - g(x)) < 1$ gives that for  a.e $x \in [0, \delta)$
 \[0 \leq  f(x) - g(x) \leq 4 \frac{p}{\la} \|\rho\|_{L^\infty} \, x^2 (f(x)-g(x)). \]
This implies that for a.e.  $x>0$ small enough $f(x) =0$ and $g(x)=0$, which means that the curves coincide.
Contradiction.
\end{proof}

\section{Curvature of minimizers} \label{app}

In \cite{Sl13, LuSl} we studied the average-distance problem considered over the set of connected
1-dimensional sets. Here we study the problem among a more restrictive set of objects, namely parameterized curves. The conditions for stationarity and regularity estimates of \cite{Sl13, LuSl}
still apply in this setting. Here we state the estimates for general $p \geq 1$, while we previously considered only $p=1$. The extension is straightforward. 

We start by stating the conditions for the case that $\mu$ is a discrete measure, $\mu = \sum_{i=1}^n m_i \delta_{x_i}$ where $m_i>0$ for all $i$ and $\sum_{i=1}^n m_i=1$. Arguing as in Lemma 7 of  \cite{Sl13} we conclude that any minimizer of  Problem \ref{main1} is a solution of a euclidean traveling salesman 
(for Problem \ref{main pen} the minimizers  were Steiner trees) and is thus a piecewise linear curve with no self-intersections (i.e. $\gamma$ is injective). Such $\gamma$ can be described as a graph as follows. Let $V$, the set of vertices, be the collection of all minimizers over $\Gamma_\gamma$ of distance to each of the point in $X=\{x_1, \dots, x_n\}$. That is let 
\begin{equation} \label{vert}
 V= \bigcup_{i=1}^n \argmin_{z \in \Gamma_\gamma} |z - x_i|. 
\end{equation}
We can write $V= \{v_1, \dots, v_m\}$ where $v$ are ordered as they appear along $\Gamma_\gamma$ (in  increasing order with respect to parameter of $\gamma$). Then $\Gamma_\gamma$ is the piecewise linear curve $[v_1, \dots, v_m]$.

For $j =1, \dots, m$ let $I_j$ be the set of indices of points in $X$ for which $v_j$ is the closest point in $V$
\begin{align} 
\begin{split}
 I_j & = \{ i \in \{1, \dots, n\} \: : \:  (\forall k=1, \dots, m) \quad d(x_i, v_j) \leq d(x_i, v_k) \} \\
      & = \{ i \in \{1, \dots, n\} \: : \:  (\forall y \in \Gamma_\gamma) \quad d(x_i, v_j) \leq d(x_i, y) \}.
\end{split}
\end{align}
If $i \in I_j$ then we say that $x_i$ {\em talks to } $v_j$. 
We say that a vertex $v_j$ is {\em tied down} if  for some $i$, $v_j =x_i$.
We then say that $v_j$ is {\em tied to} $x_i$.
A vertex which is not tied down is called  {\em free}.
As shown in \cite{Sl13}, if $x_i$ talks to $v_j$ and $v_j$ is free then 
$x_i$ cannot talk to any other vertex.

As in \cite{Sl13} we consider the optimal transportation plan between $\mu$ and its projection onto
 $\Gamma_\gamma$. That is, consider an $n$ by $m$ matrix $T$ such that 
\begin{equation} \label{otf}
  T_{ij} \geq 0,\;  \sum_{j=1}^m  T_{ij} = m_i, \te{ and }   T_{ij} > 0 \te{ implies } i \in I_j 
\end{equation}
Note that $\mu =\sum_{i=1}^n  \sum_{j=1}^m  T_{ij} \delta_{x_i}$. 
Furthermore observe that if $v_j$ is tied to $x_i$ then $i \in I_j$ and $T_{ij} = m_i$. 
Let $\pi = \sum_{i,j} m_i \delta_{x_1} \otimes \delta_{v_j}$. We note that the first marginal of $\pi$ is $\mu$ and that it describes an optimal transportation plan between $\mu$ and a measure supported on $V \subset \Gamma_\gm$.
We define $\sigma$ to be the second marginal of $\pi$ as before (above Lemma \ref{exist}). Then
$\sigma$ is a projection of $\mu$ onto the set $\Gamma_\gamma$ in that 
the mass of  $\mu$ is transported to a closest point on $\Gamma_\gamma$.  More precisely
\begin{equation} \label{def_sigma}
\sigma=  \sum_{j=1}^m \sum_{i=1}^n T_{ij} \delta_{v_j}.
\end{equation}
We note that the matrix $T$ describes an optimal transportation plan  between $\mu$ and $\sigma$ with respect to any of the transportation costs $c(x,y) = |x-y|^q$, for $q \geq 1$. 
 
We note that in this discrete setting
\begin{align} \label{dis_ene}
\begin{split}
E_\mu^{\la,\p}(\gamma) & = \sum_{i=1}^n m_i d^p(x_i, \Gamma_\gamma) + \lambda  \sum_{i=1}^{m-1} |v_{i+1}-v_i| \\
& = \sum_{j=1}^m \sum_{i \in I_j} T_{ij} |x_i-v_j|^p + \lambda \sum_{i=1}^{m-1} |v_{i+1}-v_i|
\end{split}
\end{align}

Taking the first variation in $v_j$ provides an extension to $p>1$ of
conditions for stationarity of Lemma 9 in \cite{Sl13}:
\begin{lem} \label{1st_var}
Assume that $\gamma$ minimizes $E_\mu^{\la,p}$ for discrete $\mu = \sum_{i=1}^n m_i \delta_{x_i}$. Let $V$ be the set of vertices as defined in \eqref{vert} and $T$ be any matrix (transportation plan) satisfying \eqref{otf}. Then
\begin{itemize}
\item For endpoints $j=1$ and $j=m$ let $w = v_2$ if $j=1$ and $w = v_{m-1}$ otherwise.

If $p>1$ or $v_j$ is free then
\begin{equation} \label{var_end}
\sum_{i \in I_j} p \, T_{ij} \, (x_i - v_j)|x_i - v_j|^{p-2}   + \lambda \, \frac{w-v_j}{|w-v_j|} = 0
\end{equation}
If $v_j$ is tied to $x_k$ and $p=1$  then 
\begin{equation} \label{var_end2}
\left | \sum_{i \in I_j, i \neq k}  T_{ij} \, \frac{x_i - v_j}{|x_i - v_j|}  + \lambda \, \frac{w-v_j}{|w-v_j|}
\right| \leq m_k
\end{equation}
\item If $j=2, \dots, m-1$ then if $p>1$ or $v_j$ is free 
\begin{equation} \label{var_cor}
\sum_{i \in I_j} p \, T_{ij} \, (x_i - v_j)|x_i - v_j|^{p-2}  + \lambda \, \left( \frac{v_{j-1}-v_j}{|v_{j-1}-v_j|} +  \frac{v_{j+1}-v_j}{|v_{j+1}-v_j|}  \right) = 0
\end{equation}
If $v_j$ is tied to $x_k$ and $p=1$ then
\begin{equation} \label{var_cor2}
\left|\sum_{i \in I_j, i \neq k}   T_{ij} \, \frac{x_i - v_j}{|x_i - v_j|}   + \lambda \, \left( \frac{v_{j-1}-v_j}{|v_{j-1}-v_j|} +  \frac{v_{j+1}-v_j}{|v_{j+1}-v_j|}  \right) \right| \leq m_k.
\end{equation}
\end{itemize}
\end{lem}
The proof of the lemma is analogous to one in \cite{Sl13}. 

Note that the condition at a corner provides an upper bound on the turning angle in terms of the $p-1$-st moment of the mass that talks to the corner.
These conditions can be used as in \cite{LuSl} to obtain estimates on the curvature (in the sense of a measure) of minimizers $\gamma$ of $E_\mu^{\la,\p}$ for general compactly supported measures $\mu$. 
 In particular adapting the proof of Theorem 5.1 and Lemma 5.2 of \cite{LuSl} implies:
 
 \begin{prop}  \label{reg_p}
 Let $\mu \in \M$ and let $\lambda>0$ and $p \geq 1$. If $\gm:[0,L] \to \R^d$ is an arc-length-parameterized minimizer of $\EE$ then $\gm' \in BV([0,L], \R^d)$ and
\begin{equation} \label{regest1}
 \| \gm' \|_{TV([0,L])} \leq    \frac{p}{\lambda} \,\diam(\supp(\mu))^{p-1} \, \mu(\R^d).   
\end{equation}
\end{prop}
A particular consequence of this estimate is that for all $T \in [0,L)$, $\lim_{t \to T+} \gm'(t)$ exists.
It is straightforward to prove that thus $\gm$ has a right derivative $\gm'(t+)$ for all $t \in [0,L)$. Analogous statements hold for the left derivative. 

We note that the estimate holds if we consider the same problem on the class of curves with fixed endpoints:
$\gamma : [0,L] \to \R^d$ with $\gamma(0)=P$ and $\gamma(1)=Q$ with $P,Q \in \conv(\mu)$. The proof is essentially the same.

A consequence of this observation is that we can formulate a localized version of the estimate. In particular let $\gm$ be the minimizer of $\EE$ as in the Proposition \ref{reg_p}.  Let $\pi$ and $\sigma$ be as defined above Lemma \ref{exist}.
For any interval  $I=(t,t+\delta) \subset [0,L]$
\begin{equation} \label{regest2}
 \| \gm' \|_{TV(I)} \leq   p \,\diam(\supp(\mu))^{p-1}  \frac{1}{\lambda} \,  \sigma(\gm(I)).   
\end{equation}
The estimate bounds how much can the curve $\gm$ turn within interval $I$ based on the $p-1$-st moment of the mass in $\mu$ that projects onto the set $\gm(I)$.
Let $\mu_I$ be the measure defined as $\mu_I(U) = \pi(U \times \gm(I))$, that is the $\mu$ measure of the set of points that projects onto $\gm(I)$. 
The estimate follows from Proposition \ref{reg_p} using the observation that $\gm|_{I}$ is a minimizer of $E_{\mu_I}^{\la, \p}$ among curves which start at $\gm(t)$ and end at $\gm(t+\delta)$.

 In this work we need finer information. We focus on dimension $d=2$. We need information not only on how much a curve turns but also on about the direction it turns in. 
 \begin{lem} \label{reg_U}
Consider dimension  $d=2$. Let $\mu \in \M$,  $\lambda >0$ and $p \geq 1$. Let $\gamma:[0,L] \to \R^2$ be an arc-length-parameterized minimizer of $E_\mu^{\la,\p}$. Let $ t \in [0,L)$. 
 By rotation and translation we can assume that $\gm(t)=0$, $\gamma'(t+)=e_1$. Let $I = (t,t+\delta)$ be such that $t+\delta < L$, $\| \gm' \|_{TV(I)} < \frac{1}{2}$. We define the region underneath the curve segment $\gm(I)$ to be as depicted on Figure \ref{fig:U}. That is let $\ell_1^- = \{ s (0,1) \::\: s \leq 0 \}$ and
 $\ell_2^- = \{ \gm(t+\delta) + s (\gm'((t+\delta)-))^\perp \::\: s \leq 0 \}$. Let $U$ be the connected component of 
 $\R^2 \backslash (\ell_1^- \cup \ell_2^- \cup \gm(I))$ which contains the point $\gm(t+ \frac{\delta}{2}) - \frac{\delta}{4} (0,1)$. 
 
 Let $D$ be the maximal distance of a point in $\supp(\mu) \cap U$, which talks to $\gm(I)$. That is let   $D = \sup\{ d(x,\gm(I)) \::\: x \in \supp(\mu) \cap U, \argmin_{z \in \Gamma_\gm} d(x,z) \cap \gm(I) \neq \emptyset\}$.
 Then 
 \begin{equation} 
  \sup_{s \in I} \gm'(s) \cdot e_2 \leq   \frac{p}{\la}  \,D^{p-1} \mu(U).
 \end{equation}
 \end{lem}
 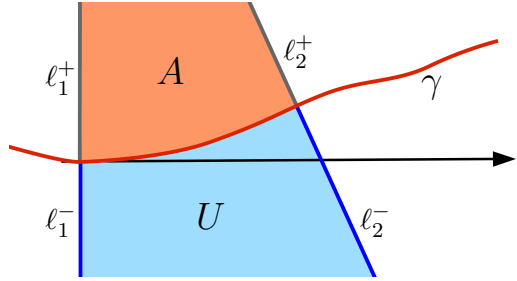
\begin{figure}[h!]
\includegraphics[scale=0.667]{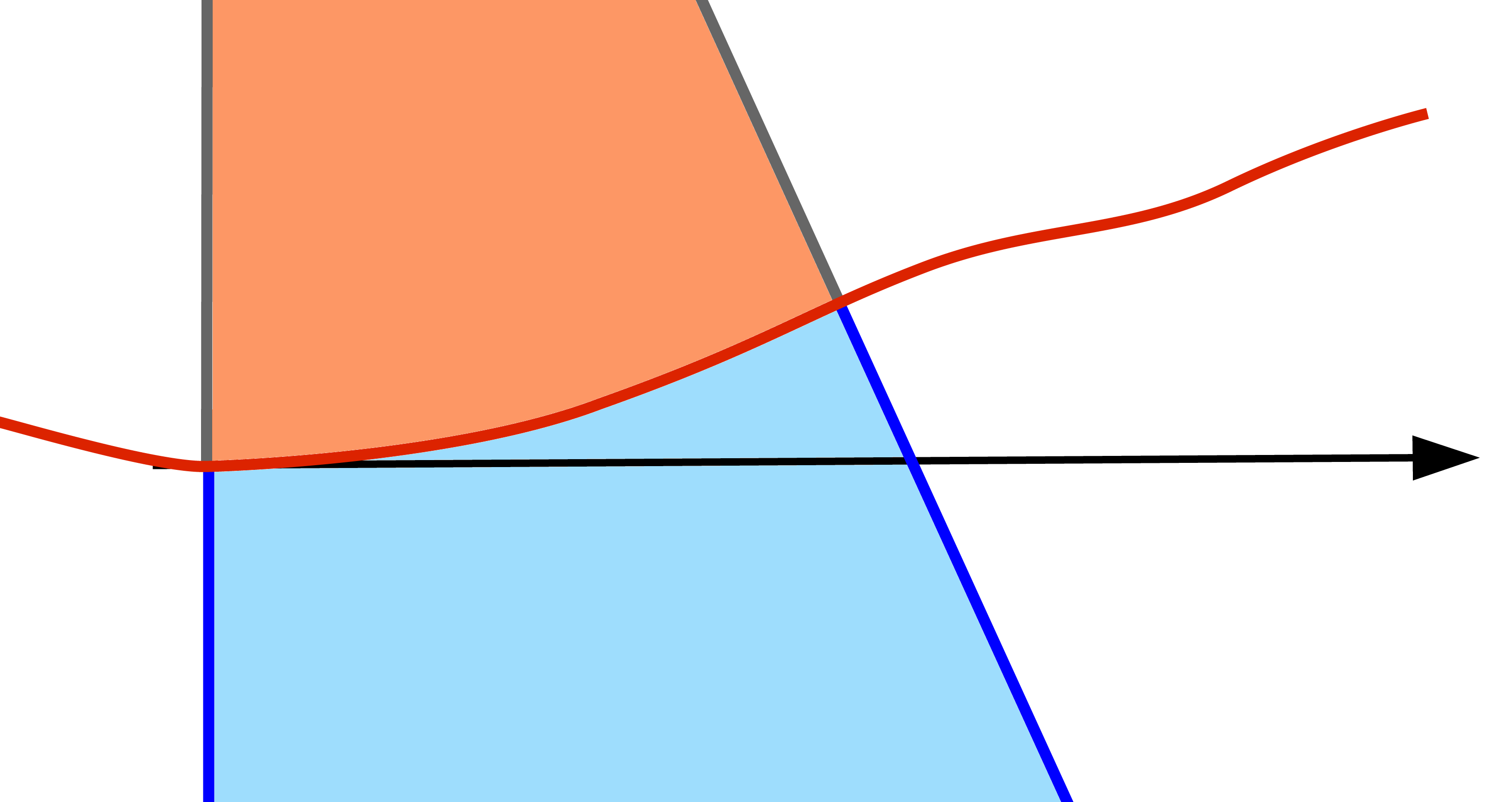}
\put(-125,18){\Large $U$}
\put(-140,73){\Large $A$}
\put(-63,18){$\ell_2^-$}
\put(-182,18){$\ell_1^-$}
\put(-92,82){$\ell_2^+$}
\put(-182,72){$\ell_1^+$}
\put(-40,70){\Large $\gm$}
\caption{The geometry of the configuration near the last double point.}
\label{fig:U}
\end{figure}
 \begin{proof}
 By approximating as in the proof of Theorem 5.1 in \cite{LuSl} the problem can be reduced to considering discrete measures. Thus we assume that $\mu = \sum_{i=1}^n m_i \delta_{x_i}$.
  
 Let $A$ be the region above the curve segment $\gm(I)$. That is let $\ell_1^+ = \{ s (0,1) \::\: s \geq 0 \}$ and
 $\ell_2^+ = \{ \gm(t+\delta) + s (\gm')^\perp(t+\delta-) \::\: s \geq 0 \}$ and let $A$ be the connected component of 
 $\R^2 \backslash (\ell_1^+ \cup \ell_2^+ \cup \gm(I))$ which contains the point $\gm(t+ \frac{\delta}{2}) + \frac{\delta}{4} (0,1)$. 
Note that all of the mass of $\mu$ that talks to $\gm(I)$ is contained in $U \cup A \cup \gamma(I)$. 
 
Due to an assumption on $I$, for all $v_j \in \gamma(I)$ the angle between $v_{j+1} - v_j$ and $e_1$ is less than $\pi/4$ and so is the angle $v_{j} - v_{j-1}$ and $e_1$. Therefore if $i \in I_j$ and $x_i \in A$ then the directed angle between $e_1$ and $x_i - v_j$ is between $\pi/4$ and $3 \pi/4$. Therefore $( v_j - x_i) \cdot e_2 < 0$.

Let us first consider the case that $p>1$. Then from \eqref{var_cor} follows that for $j$ such that $v_j \in \gm(I)$
\begin{align*}
  \sum_{i \in I_j, x_i \in U} p \, T_{ij} \, |x_i - v_j|^{p-2}   ( v_j - x_i) \cdot e_2 & \geq \sum_{i \in I_j} p \, T_{ij} \, |x_i - v_j|^{p-2}   ( v_j - x_i) \cdot e_2 \\
  & = \lambda \, \left(\frac{v_{j+1}-v_j}{|v_{j+1}-v_j|}  - \frac{v_{j}-v_{j-1}}{|v_{j}-v_{j-1}|}  \right) \cdot e_2 
\end{align*}
Consider $s \in (t, t+\delta)$.
Summing up over all  $j$ such that $v_j \in \gm((t,s))$ gives
\[ p \mu( U) D^{p-1}  \geq \lambda \gm'(s -) \cdot e_2, \]
which establishes the desired claim.

Consider now $p=1$.
From \eqref{var_cor2} follows that for $j$ such that $v_j \in \gm(I)$
\begin{align*}
  \sum_{i \in I_j, x_i \in U}  T_{ij} \, \frac{1}{|x_i - v_j|}   ( v_j - x_i) \cdot e_2 & \geq \sum_{i \in I_j} T_{ij} \, \frac{1}{|x_i - v_j|}   ( v_j - x_i)  \cdot e_2 \\
  & \geq \lambda \, \left(\frac{v_{j+1}-v_j}{|v_{j+1}-v_j|}  - \frac{v_{j}-v_{j-1}}{|v_{j}-v_{j-1}|}  \right) \cdot e_2 -\mu(v_j)
\end{align*}
 Summing over $j$ such that $v_j \in \gamma((t,s))$ again provides the desired claim.
 \end{proof}

 \bigskip
\noindent {\bf Acknowledgments.}
Both authors are thankful to FCT (grant UTA CMU/MAT/0007/2009).
XYL acknowledges the support by ICTI.
DS is grateful to  NSF (grant DMS-1211760) for its support. 
The authors would like to thank the Center for Nonlinear Analysis of the Carnegie Mellon University for its support.

\end{document}